\documentclass[12pt]{amsproc}
\newtheorem{theorem}{\sc Theorem}[section]
\newtheorem{lemma}[theorem]{\sc Lemma}

\newtheorem{remark}[theorem]{\sc Remark}

\usepackage{fancyhdr}
\begin{document}
\title{Supersolvable Frobenius groups with nilpotent centralizers}
\chead[Jhone Caldeira and Emerson de Melo]{Frobenius groups of automorphisms}

\author{Jhone Caldeira and Emerson de Melo}
\address{Instituto de Matem\'atica e Estat\'istica, Universidade Federal de Goi\'as,
Goi\^ania-GO, 74690-900 Brazil }
\email{jhone@ufg.br}
\address{Department of Mathematics, Universidade de Bras\'ilia, Bras\'ilia-DF, 70910-900 Brazil }
\email{emerson@mat.unb.br}


\begin{abstract}

Let $FH$ be a supersolvable Frobenius group with kernel $F$ and complement $H$. Suppose that a finite group $G$ admits $FH$ as a
group of automorphisms in such a manner that $C_G(F)=1$ and $C_{G}(H)$ is nilpotent of class $c$. 
We show that $G$ is nilpotent of $(c,\left|FH\right|)$-bounded class. 

\end{abstract}

\maketitle

\section{Introduction}

Let a group $A$ act by automorphisms on a group $G$. We denote by $C_G(A)$ the set 
$C_G(A)=\{x\in G ; \ x^a=x \ \textrm{for all} \ a\in A\}$, the centralizer of $A$ in $G$ 
(the fixed-point subgroup). Very often the structure of $C_G(A)$ has strong influence 
over the structure of $G$. Recently, prompted by Mazurov's problem 17.72 in the Kourovka 
Notebook \cite{KourNot}, 
some attention was given to the situation where a Frobenius group $FH$ acts by 
automorphisms on a finite group $G$. Recall that a Frobenius group $FH$ with kernel 
$F$ and complement $H$ can be characterized 
as a finite group that is a semidirect product of a normal subgroup $F$ by $H$ such that 
$C_F(h) =1$ for every $ h \in H \setminus \{1\}$. By Thompson's theorem \cite{Tho} the kernel $F$ 
is nilpotent, and by Higman's theorem \cite{Hig} the nilpotency class of $F$ 
is bounded in terms of the least prime divisor of $|H|$.

In the case where the kernel $F$ acts fixed-point-freely on $G$, 
some results on the structure of $G$ were obtained by Khukhro, Makarenko and Shumyatsky 
in \cite{KhuMakShu}. The authors prove that various properties of $G$ are in a certain 
sense close to the corresponding properties of its subgroup $C_G(H)$, possibly also 
depending on $H$. In particular, they proved the following result.

\begin{theorem}[\cite{KhuMakShu}, Theorem 2.7 (c)]\label{auxthe}
Suppose that a finite group $G$  admits a Frobenius group of automorphisms $FH$ with kernel $F$ 
and complement $H$ such that $C_G(F) = 1$ and $C_G(H)$ is nilpotent. Then $G$ is nilpotent.
\end{theorem}

Under the additional assumption that the Frobenius group of automorphisms $FH$ is metacyclic, 
that is, supposing that the kernel $F$ is cyclic,  
the authors use some Lie ring methods to obtain upper bounds for the nilpotency class of $G$. 
They proved that if $C_G(H)$ is nilpotent of class $c$, then the nilpotency class of $G$ is 
bounded in terms of $c$ and $|H|$. In the case when $GF$ is also a Frobenius group with kernel $G$ 
and complement $F$ (so that $GFH$ is a double Frobenius group) the latter result was obtained 
earlier in \cite{MakShu}. Examples in \cite{KhuMakShu} show that the result on the  
nilpotency class of $G$ is no longer true in the case of non-metacyclic Frobenius groups. 

Recall that a group $N$ is supersolvable if $N$ possesses a normal series with cyclic factors 
such that each term is normal in $N$. It is easy to see that a supersolvable group possesses a 
chief series whose factors have prime order.

Throughout the paper we use the expression ``$(a,b,\dots )$-bounded'' to mean 
``bounded from above in terms of  $a, b,\dots$ only''. In the present paper we consider the situation in which a not necessarily metacyclic Frobenius group acts by automorphisms on a finite group. More precisely, we prove the following theorem.

\begin{theorem}\label{mainthe1}
Let $FH$ be a supersolvable Frobenius group with kernel $F$ and complement $H$. 
Suppose that $FH$ acts on a finite group $G$ in such a way that $C_G(F)=1$ 
and $C_{G}(H)$ is nilpotent of class $c$. Then $G$ is nilpotent 
of $(c,\left|FH\right|)$-bounded class.
\end{theorem}

Note that $G$ in Theorem \ref{mainthe1} is nilpotent by Theorem \ref{auxthe}. 
We wish to show that the nilpotency class of $G$ is $(c,\left|FH\right|)$-bounded. It should be mentioned that in the case of metacyclic $FH$, Khukhro, Makarenko and Shumyatsky in \cite[Theorem 5.1]{KhuMakShu} give a bound independent of the order $|F|$. At present it is unclear if in Theorem \ref{mainthe1} the bound can be made independent of $|F|$. The proof is based on an analogous result on Lie algebras.

In \cite{CalMelShu,KhuShu} we can find some other 
results bounding the nilpotency class for groups 
acted on by Frobenius groups of automorphisms with non-cyclic kernels.

\section{Some results on graded Lie algebras}\label{secgrLiealg}

Let $A$ be an additively written abelian group and $L$ an $A$-graded Lie algebra:

$$L=\bigoplus_{a\in A}L_a \ \textrm{, } \ [L_a,L_b]\subseteq L_{a+b}.$$

The particular case of $\mathbb{Z}/n\mathbb{Z}$-graded Lie algebras arises naturally in the study of Lie 
algebras admitting an automorphism $\varphi$ of order $n$. This is due to the fact that, 
after the ground field is extended by a primitive $n$th root of unity $\omega$, 
the eigenspaces $L_i=\{x \in L \ ; \  x^{\varphi}=\omega^ix\}$ behave like the 
components of a $\mathbb{Z}/n\mathbb{Z}$-grading. For example, the proof 
of a well-known theorem of Kreknin \cite{kre} stating that a Lie ring $L$ admitting a 
fixed-point-free automorphism of finite order $n$ is soluble with 
derived length bounded by a function of $n$, is reduced to proving 
the solvability of a $\mathbb{Z}/n\mathbb{Z}$-graded Lie ring with $L_0 = 0$. 

\begin{theorem}\label{solthe}
Let $n$ be a positive integer and $L$ be a $\mathbb{Z}/n\mathbb{Z}$-graded Lie algebra. 
If $L_0=0$, then $L$ is solvable of $n$-bounded derived length.
\end{theorem}

In the special case where $n$ is a prime we have the following well-known result proved by Higman.

\begin{theorem}\cite{Hig}\label{nilthe}
Let $p$ be a prime number and $L$ be a $\mathbb{Z}/p\mathbb{Z}$-graded Lie algebra. 
If  $L_0=0$, then $L$ is nilpotent of $p$-bounded class.
\end{theorem}

If $J, Y,J_1,\ldots, J_s$ are subsets of $L$ we use $[J,Y]$ to denote the subspace 
of $L$ spanned by the set $\{[j,y] \ ; \ j\in J,y\in Y\}$ and for $i \geq 2$ we write 
$[J_1,\ldots,J_i]$ for $[[J_1,\ldots,J_{i-1}],J_i]$.

The next two results are also criteria for solvability and nilpotency of graded Lie algebra, respectively.

\begin{theorem}\cite[Theorem 1]{Khu}\label{Khuthe} Let $n$ be a positive integer 
and $L$ be a $\mathbb{Z}/n\mathbb{Z}$-graded Lie algebra. Suppose that 
$[L,\underbrace{L_0,\ldots,L_0}_m]=0$. Then, $L$ is solvable of $(n,m)$-bounded derived length.
\end{theorem}

\begin{theorem}\cite[Proposition 2]{KhuShu}\label{Grad2}
Let $A$ be an additively written abelian group and $L$ an $A$-graded Lie algebra. 
Suppose that there are
at most $d$ nontrivial grading components among the $L_a$. If  
$[L,\underbrace{L_a,\ldots,L_a}_m]=0$ for all $a\in A$, then $L$ is 
nilpotent of $(d,m)$-bounded class.
\end{theorem}

Now, let $p$ be a prime number and $L$ be a $\mathbb{Z}/p\mathbb{Z}$-graded 
Lie algebra. Assume that there exist non-negative integers $u$ and $v$ such that 
\begin{equation}\label{eq1}
[L,\underbrace{L_0,\ldots,L_0}_u]=0
\end{equation}
\noindent and
\begin{equation}\label{eq2}
[[L,L] \cap L_0,\underbrace{L_a,\ldots,L_a}_{v}]=0, \ \mbox{for all} \ a \in \mathbb{Z}/p\mathbb{Z}.
\end{equation}

We finish this section showing that conditions (\ref{eq1}) and (\ref{eq2}) together are 
sufficient to conclude that $L$ is nilpotent of $(p,u,v)$-bounded class. By Theorem 
\ref{Khuthe}, condition (\ref{eq1}) implies that $L$ is solvable with $(p,u)$-bounded 
derived length. Thus, we can use induction on the derived 
length of $L$. If $L$ is abelian, there is nothing to prove. Assume that $L$ is metabelian. 
In this case, $[x, y, z] = [x, z, y]$, for every $x \in [L, L]$ and $y,z \in L$.

For each $b \in \mathbb{Z}/p\mathbb{Z}$ we denote $[L,L] \cap L_b$ by $L_b'$. 
By Theorem \ref{Grad2} it is sufficient to prove that there exists a $(p,u,v)$-bounded number $t$ such that

$$[L_b',\underbrace{L_a,\ldots,L_a}_{t}]=0, \ \mbox{for all} \ a,b \in \mathbb{Z}/p\mathbb{Z}.$$

If $b=0$, this follows from (\ref{eq2}) with $t=v$. Suppose that $b \neq 0$.
If $a=0$, the commutator is zero from (\ref{eq1}) with $t=u$. In the case where $a \neq 0$ 
we can find a positive integer $s < p$ such that $b + sa = 0 \ (\textrm{mod }p)$. Therefore,
we have $[L_b',\underbrace{L_a,\ldots,L_a}_{v + s}] \subseteq [L_0',\underbrace{L_a,\ldots,L_a}_v]$. 
We know that the latter commutator is 0 by (\ref{eq2}). 

These previous ideas were applied in a similar way in \cite{Eme1} and \cite{Eme2}. 

\section{Bounding nilpotency class of Lie algebras}

In this section $FH$ denotes a finite supersolvable Frobenius group with kernel $F$ 
and complement $H$.

%
%
%

Let $R$ be an associative ring with unity. Assume that the characteristic 
of $R$ is coprime with $|F|$ and the additive group of $R$ is finite (or locally finite). 
Let $L$ be a Lie algebra over $R$. 
The main goal of this section is to prove the following theorem.

\begin{theorem}\label{Liealgthe}
Suppose that $FH$ acts by automorphisms on $L$ in such a way that 
$C_L(F)=0$ and $C_{L}(H)$ is nilpotent of class $c$. Then $L$ is nilpotent 
of $(c,\left|FH\right|)$-bounded class.
\end{theorem}

If $F$ is cyclic, then 
$L$ is nilpotent of $(c,\left|H\right|)$-bounded class by \cite[Theorem 5.1]{KhuMakShu}. 
Therefore without loss of generality we may assume that $F$ is not a cyclic group.

Let $Z$ be a subgroup of prime order $p$ of $Z(F)$ such that $Z \triangleleft FH$ and let $\varphi$ be a generator of $Z$.

Now, let $\omega$ be a primitive $p$th root of unity. We extend the ground ring of $L$ by 
$\omega$ and denote by $\tilde{L}$ the algebra $L\otimes_{\mathbb{Z}}\mathbb{Z}[\omega]$. 
The action of $FH$ on $L$ extends naturally to $\tilde{L}$. Note that $C_{\tilde{L}}(F)=0$ and 
$C_{\tilde{L}}(H)$ is nilpotent of class $c$. Since $L$ and $\tilde{L}$ have the same nilpotency class, 
it is sufficient to bound the class of $\tilde{L}$. Hence, without loss of generality 
it will be assumed that the ground ring contains $\omega$ so that we will work with $L$ 
rather than $\tilde{L}$.

For each $i=0,\ldots,p-1$ we denote by $L_i$ the  $\varphi$-eigenspace corresponding 
to eigenvalue $\omega^i$, that is,  $L_i=\{x \in L \ ; \  x^{\varphi}=\omega^ix\}$. 
We have  $$L=\bigoplus_{i=0}^{p-1}L_i \ \textrm{ and } \ [L_i,L_j]\subseteq L_{i+j(\textrm{mod} \ p)}. $$ 

Note that  $L_0=C_L(Z)$. It is easy to see that $FH/Z$ acts on $L_0$ in such a manner 
that $C_{L_0}(F/Z)=0$. Thus, by induction on $|FH|$ we conclude that $L_0$ is 
nilpotent of $(c,|FH|)$-bounded class.

A proof of the next lemma can be found in \cite[Lemma 2.4]{KhuMakShu}. 
It will be useful to decompose the subalgebra $L_0$ into the fixed-points of the subgroups $H^f$, $f\in F$.

\begin{lemma}\label{lemma dec Frb}
Suppose that a finite group $N$ admits a Frobenius group of automorphisms $KB$ with 
kernel $K$ and complement $B$ such that $C_N(K)=1$. Then $N=\langle C_N(B)^y ; \ y \in K \rangle$.
\end{lemma}

In the above lemma we can write $N=\langle C_N(B^y) ; \ y \in K \rangle$, 
since $C_N(B)^y=C_N(B^y)$.

For any $f \in F$ we denote by $V_f$ the subalgebra $C_L(ZH^f)\leq L_0$. 
Note that $ZH^f$ is also a Frobenius group and that $ZH^f=ZH$ whenever $f \in Z$. 

\begin{lemma}\label{lemma dec L0}
We have $L_0 = \sum_{f\in F} V_f$.
\end{lemma}
\begin{proof}
The subalgebra $L_0$ is $FH$-invariant and so it admits the natural action by the group $FH/Z$. 
Moreover, $F/Z$ acts fixed-point-freely on $L_0$. Since the additive group of $L$ is finite, 
it is immediate from Lemma \ref{lemma dec Frb} that 
$L_0 =\langle C_{L_0}(H^{\tilde{f}}) ; \ \tilde{f} \in F/Z \rangle$. As a 
consequence we have that $L_0 =\langle C_{L}(ZH^f) ; \ f \in F \rangle$.
\end{proof}

It is clear that any conjugated $H^f$ of $H$ can be considered as a Frobenius 
complement of $FH$. Now we describe the action $H^f$ 
on the homogeneous components $L_i$.  Since $H$ is cyclic, we can choose 
a generator $h$ of $H$ and $r \in \{1,2,\ldots,p-1\}$ such that $\varphi^{h^{-1}} = \varphi^r$. 
Then $r$ is a primitive $q$th root of 1 in 
$\mathbb{Z}/p\mathbb{Z}$. The group $H$ permutes the homogeneous components $L_i$ as follows: 
$L_i^h = L_{ri}$ for all $i \in \mathbb{Z}/p\mathbb{Z}$. Indeed, if 
$x_i \in L_i$, then $(x_i^h)^\varphi = x_i^{h\varphi h^{-1}h} = (x_i^{\varphi^r})^h = 
\omega^{ri}x_i^h$. On the other hand, since $F$ commutes with $\varphi$, we also 
have $L_i^{h^f} = L_{ri}$ for all 
$i \in \mathbb{Z}/p\mathbb{Z}$ (because if $x_i \in L_i$, then 
$(x_i^{h^f})^\varphi = \omega^{ri}x_i^{h^f}$). Thus, the action of $H^f$ 
on the components $L_i$ coincides with the action of $H$. 

To lighten the notation we establish the following convention.

\begin{remark}[Index Convention]
In what follows, for a given $u_s \in L_s$ we denote both $u_s^{h^i}$ and 
$u_s^{(h^f)^i}$ by $u_{r^is}$, 
since $L_s^{h^i} = L_s^{(h^f)^i}=L_{r^is}$. Therefore, we can write 
$u_s + u_{rs} + \cdots + u_{r^{q-1}s}$ to mean a fixed-point of $H^f$ for any $f\in F$.
\end{remark}

\begin{lemma}\label{L_0}
There exists a $(c,|FH|)$-bounded number $u$ such that $[L,\underbrace{L_0,\ldots,L_0}_u]=0$.
\end{lemma}
\begin{proof} It suffices to prove that $[L_b,\underbrace{L_0,\ldots,L_0}_u]=0$ for any
$b \in \{0,1,\ldots, p-1\}$.

Taking into account that $L_0$ is nilpotent of $(c,|FH|)$-bounded class, we may assume $b \neq 0$.

By Lemma \ref{lemma dec L0}, 
$L_0 =  \sum_f V_f$. Let $u_b \in L_b$. Since $u_b + u_{rb} + \cdots + u_{r^{q-1}b} \in C_L(H^f)$ and 
$V_f \subseteq C_L(H^f)$, we have
$$[u_b + u_{rb} + \cdots + u_{r^{q-1}b}, \underbrace{V_f,\ldots,V_f}_c] = 0.$$
Thus, $\sum_{i=0}^{q-1} [u_{r^ib}, \underbrace{V_f,\ldots,V_f}_c] = 0$. 
On the other hand, $[u_{r^ib}, \underbrace{V_f,\ldots,V_f}_c] \in L_{r^ib}$ 
and $L_{r^ib}\neq L_{r^jb}$ whenever $i \neq j$. Therefore, we obtain $[L_b, \underbrace{V_f,\ldots,V_f}_c] = 0$.

Let $S$ be an $FH$-invariant subalgebra of $L_0$ and let $S_f = S \cap V_f$ for $f \in F$. 
It follows that $S = \sum_f S_f$. Now, we will show that there exists a $(c,|FH|)$-bounded 
number $u$ such that $[L_b,\underbrace{S,\ldots,S}_u]=0$ using induction on the 
nilpotency class of $S$. Since $[S,S]$ is nilpotent of smaller class, 
there exists a $(c,|FH|)$-bounded number $u_1$ such that $[L_b,\underbrace{[S,S],\ldots,[S,S]}_{u_1}]=0$. 

Now, set $l=(c-1)\left|F\right|+1$ and 
$W=[L_b,S_{i_1},\ldots,S_{i_l}]$ for some choice of $S_{i_j}$ in $\{S_{f} ; \ f\in F \}$. 
It is clear that for any permutation $\pi$ of the symbols $i_1,\ldots,i_l$ we have 
$W\leq [L_b,S_{\pi(i_1)},\ldots,S_{\pi(i_l)}]+[L_b,[S,S]]$. Also, note that the number $l$ 
is big enough to ensure that some $S_i$ occurs in the list $S_{i_1},\ldots,S_{i_l}$ 
at least $c$ times. Thus, we deduce that 
$W\leq [[L_b,\underbrace{S_i,\ldots,S_i}_{c}], *,\ldots,*]+[L_b,[S,S]]$, 
where the asterisks denote some of the subalgebras $S_j$ which, in view of the fact that 
$[L_b,\underbrace{S_i,\ldots,S_i}_{c}]=0$, are of no consequence. Hence, $W\leq [L_b,[S,S]]$.

Further, for any choice of $S_{i_1},\ldots,S_{i_l} \in \{S_{f} ; \ f \in F \}$ the same 
argument shows that $$[W,S_{i_1},\ldots,S_{i_l}]\leq [W,[S,S]]\leq [L_b,[S,S],[S,S]].$$ 
More generally, for any $m$ and any $S_{i_1},\ldots,S_{i_{ml}} \in \{S_f; \ f \in F\}$ 
we have $$[L_b,S_{i_1},\ldots,S_{i_{ml}}]\leq [L_b,\underbrace{[S,S],\dots,[S,S]}_m].$$

Put $u=u_1l$. The above shows that 
$$[L_b,S_{i_1},\ldots,S_{i_{u}}]\leq [L_b,\underbrace{[S,S],\dots,[S,S]}_{u_1}]=0.$$ 
Of course, this implies that $[L_b,\underbrace{S,\ldots,S}_u]=0$. 
The lemma is now straightforward from the case where $S=L_0$.
\end{proof}

\noindent\textbf{Proof of Theorem \ref{Liealgthe}} \\

In view of Lemma \ref{L_0}, applying the arguments of Section \ref{secgrLiealg}, 
Theorem \ref{Liealgthe} holds if we show the following:

\begin{lemma}\label{metablemma} Let $L$ be metabelian. There exists a 
$(c,\left|FH\right|)$-bounded number $v$ such that $[[L,L] \cap L_0,\underbrace{L_a,\ldots,L_a}_{v}]=0, 
\ \mbox{for all} \ a \in \mathbb{Z}/p\mathbb{Z}$.
\end{lemma}
\begin{proof}
Recall that $[L,L]\cap L_0 =  \sum_f [L,L]\cap V_f$ where $V_f=C_L(ZH^f)$. Set 
$V=C_L(ZH)$ and $V'=[L,L]\cap V$. First we prove that $[V',\underbrace{L_a,\ldots,L_a}_{v}]=0$ for any 
$a \in \mathbb{Z}/p\mathbb{Z}$.

For any $u_a \in L_a$ we have $u_a + u_{ra} + \cdots + u_{r^{q-1}a} \in C_L(H)$ (under Index Convention). Thus, 
$$[V',\underbrace{u_a + u_{ra} + \cdots + u_{r^{q-1}a},\ldots,v_a + v_{ra} + \cdots + v_{r^{q-1}a}}_c] = 0$$
for any $c$ elements $u_a,\ldots,v_a \in L_a$.

Let $T$ denote the span of all the sums $x_a + x_{ra} + \cdots + x_{r^{q-1}a}$ over $x_a \in L_a$ 
(in fact, $T$ is the fixed-point subspace of $H$ on $\bigoplus_{i=0}^{q-1}L_{r^ia}$). Then the 
latter equality means that
$$[V',\underbrace{T,\ldots,T}_c] = 0.$$

Applying $\varphi^j$ we also obtain
$$[V',\underbrace{T^{\varphi^j},\ldots,T^{\varphi^j}}_c] = 
[(V')^{\varphi^j},\underbrace{T^{\varphi^j},\ldots,T^{\varphi^j}}_c] = 0.$$

A Vandermonde-type linear algebra argument shows that $L_a \subseteq \sum_{j=0}^{q-1}T^{\varphi^j}$. 
Actually this fact is a consequence of the following result:

\begin{lemma}\cite[Lemma 5.3]{KhuMakShu}\label{auxlemma} Let $\left\langle \alpha\right\rangle$ 
be a cyclic group of order $n$, and $\omega$ a primitive nth root of unity. 
Suppose that $M$ is a $\mathbb{Z}[\omega]\left\langle\alpha\right\rangle$-module 
such that $M = \sum_{i=1}^{m} M_{t_i}$,
where $x\alpha = \omega^{t_i}x$ for $x \in M_{t_i}$ and 
$0 \leq t_1 < t_2 < \cdots < t_m < n$. 
If $z = y_{t_1} + y_{t_2} + \cdots + y_{t_m}$ for $y_{t_i} \in M_{t_i}$, 
then for some $m$-bounded number $d_0$ every element $n^{d_0}y_{t_s}$ is 
a $\mathbb{Z}[\omega]$-linear combination of the elements 
$z, z\alpha, \ldots, z\alpha^{m-1}$.
\end{lemma}

Now we can apply Lemma \ref{auxlemma} with $\alpha = \varphi$, 
$M = L_a + L_{ra} + \cdots + L_{r^{q-1}a}$ 
and $m = q$ to $w = u_a + u_{ra} + \cdots + u_{r^{q-1}a} \in T$ for any $u_a \in L_a$, 
because here the indices $r^ia$ can be regarded as pairwise distinct residues modulo 
$p$ ($r$ is a primitive $q$th root of 1 modulo $p$). Since $p$ is invertible in our 
ground ring, follows that $L_a \subseteq \sum_{j=0}^{q-1}T^{\varphi^j}$.

Set $v = (c-1)q + 1$. We now claim that $[V',\underbrace{L_a,\ldots,L_a}_v]=0$. Indeed, after replacing 
$L_a$ with $\sum_{j=0}^{q-1}T^{\varphi^j}$ and expanding the sums, in each commutator 
of the resulting linear combination we can freely permute the entries $T^{\varphi^j}$, since 
$L$ is metabelian. Since there are sufficiently many of them, we can place at least $c$ 
of the same $T^{\varphi^{j_0}}$ for some $j_0$ right after $V'$ at the beginning, which gives 0.

Note that in the case where $f \in F\setminus Z$ we can consider the Frobenius 
group $ZH^f$ and conclude in a similar way that $$[[L,L]\cap V_f,\underbrace{u_a,\ldots,v_a}_v] = 0$$
for any $v$ elements $u_a,\ldots,v_a \in L_a$.

\end{proof}

\section{Proof of Theorem \ref{mainthe1}}

We will use the following results.

\begin{lemma}\label{copr nuc}
Let $G$ be a finite $p$-group admitting a nilpotent group of automorphisms 
$F$ such that $C_G(F)=1$. Let $F_{p'}$ be the Hall $p'$-subgroup of $F$. 
Then $ C_G(F_{p'})=1$.
\end{lemma}
\begin{proof}
The subgroup $C_G(F_{p'})$ is $F$-invariant, and so, it admits the natural action by the 
$p$-group $F/F_{p'}$. Since a finite $p$-group cannot act without nontrivial fixed points 
on another $p$-group, we must have $C_G(F_{p'})=1$.
\end{proof}

\begin{lemma}\cite[Lemma 2.2]{KhuMakShu}\label{Ker}
Let $G$ be a finite group admitting a nilpotent group of automorphisms 
$F$ such that $C_G(F)=1$. If $N$ is an $F$-invariant normal subgroup 
of $G$, then $C_{G/N}(F)=1$.
\end{lemma}
\begin{lemma} \cite[Theorem 2.3\label{Com}]{KhuMakShu} 
Suppose that a finite group $G$ admits a Frobenius group of 
automorphisms $FH$ with kernel $F$ and complement $H$. 
If $N$ is an $FH$-invariant normal subgroup of $G$ such that $C_N(F)=1$, then 
$C_{G/N}(H)=C_G(H)N/N$. 
\end{lemma}

We know that $G$ in Theorem \ref{mainthe1} is nilpotent. We wish to show that 
the nilpotency class of $G$ is $(c,|FH|)$-bounded. 
It is easy to see that without loss of generality we may assume that $G$ is a $p$-group. Moreover, 
by Lemma \ref{copr nuc} we also may assume that $(|G|,|F|)=1$.

Consider the associated Lie ring of the group 
$G$ $$L(G)=\bigoplus_{i=1}^{n}\gamma_i/\gamma_{i+1},$$ where $n$ is the nilpotency class 
of $G$ and $\gamma_i$ are the terms of the lower central series of $G$. The nilpotency class 
of $G$ coincides with the nilpotency class of $L(G)$. The action of the group $FH$ on $G$ 
induces naturally an action of $FH$ on $L(G)$. Since $F$ acts fixed-point-freely 
on every quotient $\gamma_i/\gamma_{i+1}$, it follows by Lemma \ref{Ker} that $C_{L(G)}(F)=0$. 
We observe that the subring $C_{L(G)}(H)$ is nilpotent of class at most $c$ by Lemma \ref{Com}. 
Theorem \ref{Liealgthe} now tells us that $L(G)$ is nilpotent 
of $(c,|FH|)$-bounded class. The proof is complete.

\end{document}